\documentclass[11pt,reqno,a4paper]{amsart}

\usepackage{a4wide}

\usepackage[T1]{fontenc}
\usepackage[utf8]{inputenc}
\usepackage{lmodern}
\usepackage{microtype}
\usepackage{amsmath,amssymb,amsthm,mathtools}
\usepackage{url}
\usepackage[colorlinks=true,linkcolor=blue,citecolor=blue,urlcolor=blue]{hyperref}

\numberwithin{equation}{section}

\newtheorem{theorem}{Theorem}[section]
\newtheorem{lemma}[theorem]{Lemma}
\newtheorem{proposition}[theorem]{Proposition}
\newtheorem{corollary}[theorem]{Corollary}
\newtheorem*{thmA}{Theorem A}
\newtheorem*{thmB}{Theorem B}
\theoremstyle{definition}

\newtheorem{question}[theorem]{Question}
\theoremstyle{remark}
\newtheorem{remark}[theorem]{Remark}

\newcommand{\pc}{\pi\chi}
\newcommand{\pcl}{\pi\chi_{\mathrm l}}
\newcommand{\pcc}{\pi\chi_{\mathrm c}}
\newcommand{\pcd}{\pi\chi_{\mathrm d}}
\newcommand{\FU}{\mathcal F_U}
\newcommand{\G}{\mathcal G}
\newcommand{\Fsum}{\mathcal H}
\newcommand{\dcup}{\mathbin{\dot\cup}}
\newcommand{\modfin}{\subseteq^{*}}
\newcommand{\nmodfin}{\mathrel{\nsubseteq^{*}}}

\title[Three local \texorpdfstring{$\pi$}{pi}-characters]{Three local \texorpdfstring{$\pi$}{pi}-characters}

\hypersetup{
  pdftitle={Three local pi-characters},
  pdfauthor={Tomasz Kania},
  pdfsubject={Linked, centred, and decreasing local pi-characters}
}

\author[T. Kania]{Tomasz Kania}
\address[T.~Kania]{Mathematical Institute\\Czech Academy of Sciences\\\v Zitn\'a 25 \\115 67 Praha 1\\Czech Republic  and  Institute of Mathematics and Computer Science\\ Jagiellonian University\\ {\L}ojasiewicza 6, 30-348 Krak\'{o}w, Poland
}
\email{kania@math.cas.cz, tomasz.marcin.kania@gmail.com}
\thanks{RVO: 67985840.}

\subjclass[2020]{Primary 54A25, 54A35; Secondary 03E17, 54D10}
\keywords{local \(\pi\)-base; linked local \(\pi\)-base; centred local \(\pi\)-base; decreasing local \(\pi\)-base; one-point filter space; pseudointersection number}

\date{July 30, 2026}

\begin{document}

\begin{abstract}
We answer in ZFC a question of N.~A. Carlson by constructing a
zero-dimensional Tychonoff space \(X\), with exactly one non-isolated point,
for which every point admits a~decreasing local \(\pi\)-base and
\[
        \pc(X)=\pcl(X)=\omega<\pcc(X)=\omega_1<\pcd(X)=\omega_2.
\]
A preliminary chain-refinement lemma shows that every inclusion-chain of
subsets of a set has a reverse-well-ordered coinitial subchain of controlled
cardinality.  Thus the result is independent of whether a decreasing local
\(\pi\)-base is understood as an inclusion-chain, a strictly decreasing
transfinite sequence, or a weakly decreasing transfinite sequence.  The
construction combines a countable star filter,
whose centred local \(\pi\)-character is at least the pseudointersection
number \(\mathfrak p\), with a matrix filter that gives the exact values
\(\omega_1\) and \(\omega_2\).
\end{abstract}

\maketitle

\section{Introduction}

A local \(\pi\)-base at a point \(p\) of a space \(X\) is a family of
non-empty open subsets of \(X\) such that every neighbourhood of \(p\)
contains one of its members.  Unlike a neighbourhood base, its members need
not contain \(p\).  Local \(\pi\)-bases are a standard tool in the theory of
cardinal functions; see, for example, \cite{Juhasz} and the recent survey
\cite{CarlsonSurvey}.

Carlson \cite{CarlsonCentered} considers three refinements.  A local
\(\pi\)-base is \emph{linked} if any two of its members intersect and
\emph{centred} if every finite subfamily has non-empty intersection.  He also
defines the decreasing local \(\pi\)-character without imposing an explicit
indexing convention.  We initially interpret a decreasing local
\(\pi\)-base as an inclusion-chain.  Theorem~B shows that this agrees with
both the strict and the weak transfinite-sequence conventions.

We denote the ordinary local \(\pi\)-character by \(\pc(p,X)\).  We write
\(\pcl(p,X)\), \(\pcc(p,X)\), and \(\pcd(p,X)\) for the least
infinite cardinalities of linked, centred, and decreasing local
\(\pi\)-bases at \(p\), respectively.  Thus a finite witness gives value
\(\omega\).  The first two invariants are always defined: the family of all open
neighbourhoods of \(p\) is itself a centred local \(\pi\)-base.  By
contrast, \(\pcd(p,X)\) need not be defined.  We write \(\pcd(X)\) only
when every point of
\(X\) admits a decreasing local \(\pi\)-base, and then take the supremum of
the pointwise invariants, as for \(\pcl(X)\) and \(\pcc(X)\).  Whenever the
last invariant is defined,
\[
        \pc(X)\leqslant\pcl(X)\leqslant\pcc(X)\leqslant\pcd(X).
\]
Carlson concludes his paper with the following question:
\begin{center}
``Is there an example of a Hausdorff space \(X\) such that
\(\pcl(X)<\pcc(X)<\pcd(X)\)?''
\end{center}
This is Question~5.4 of \cite{CarlsonCentered}; in particular, it concerns
the global invariants.

Our main result gives an unconditional answer and produces a space with only
one non-isolated point.

\begin{thmA}
In ZFC there is a zero-dimensional Tychonoff space \(X\) of cardinality
\(\omega_2\), with exactly one non-isolated point, such that every point of
\(X\) admits a decreasing local \(\pi\)-base and
\[
        \pc(X)=\pcl(X)=\omega<\pcc(X)=\omega_1<\pcd(X)=\omega_2.
\]
Consequently, Carlson's Question~5.4 has a positive answer in ZFC.
\end{thmA}

The point that removes any additional set-theoretic hypothesis is the
following general fact about chains.

\begin{thmB}
Let \(X\) be a space and let \(p\in X\).  Every local \(\pi\)-base at
\(p\) which is linearly ordered by inclusion contains a coinitial local
\(\pi\)-base that can be written as a strictly decreasing transfinite
sequence.  It may be chosen to have cardinality at most
\(\max\{\omega,|X|\}\), and no larger than the original family.
Consequently, the following three conventions for a decreasing local
\(\pi\)-base give the same invariant: an arbitrary inclusion-chain, a
strictly decreasing transfinite sequence, and a weakly decreasing
transfinite sequence.

Whenever \(\pcd(p,X)\) is defined, one has
\(\pcc(p,X)\leqslant\pcd(p,X)\), because an inclusion-chain of non-empty
sets is centred.  If \(X\) is countable and \(\pcd(p,X)\) is defined,
then \(\pcd(p,X)=\omega\).  Hence, for countable \(X\), the inequality
\(\pcc(p,X)>\omega\) implies that \(\pcd(p,X)\) is undefined.
\end{thmB}

The space used for Theorem~A is the sum of two one-point filter spaces.  The
first is countable.  Its stars form a linked local \(\pi\)-base, while its
centred local \(\pi\)-character is at least \(\mathfrak p\).  Theorem~B then
excludes every decreasing local \(\pi\)-base in this component.  The second
component is a matrix filter whose centred and decreasing local
\(\pi\)-characters are \(\omega_1\) and \(\omega_2\), respectively.
Projection lemmas for sums of filters make the two obstructions coexist at a
single point.

\section{Chains and one-point filter spaces}

For families \(\mathcal A\) and \(\mathcal C\) of subsets of the same set,
we say that \(\mathcal A\) is \emph{coinitial} in \(\mathcal C\) if every
member of \(\mathcal C\) contains a member of \(\mathcal A\).  We do not
require \(\mathcal A\subseteq\mathcal C\).

A chain of sets is called \emph{reverse well ordered} if it has an
enumeration \(\langle C_\alpha:\alpha<\theta\rangle\) such that
\(C_\beta\subsetneq C_\alpha\) whenever \(\alpha<\beta<\theta\).

\begin{lemma}\label{lem:chain-refinement}
Let \(S\) be a set and let \(\mathcal C\subseteq\mathcal P(S)\) be a
non-empty family linearly ordered by inclusion.  Then \(\mathcal C\) has a
coinitial reverse-well-ordered subchain.  If \(S\) is infinite, this
subchain may be chosen with cardinality at most \(|S|\); if \(S\) is finite,
it may be chosen finite.  In particular, every strictly decreasing
transfinite sequence of subsets of an infinite set \(S\) has length less
than \(|S|^+\).
\end{lemma}

\begin{proof}
If \(S\) is finite, then \(\mathcal C\) is finite and has a least member,
which by itself is coinitial.  Suppose that \(S\) is infinite and put
\(\kappa=|S|\).

Construct members of \(\mathcal C\) by transfinite recursion.  Having chosen
\(\langle C_\beta:\beta<\alpha\rangle\), stop if the selected family is
coinitial in \(\mathcal C\).  Otherwise there is \(C\in\mathcal C\) which
contains no previously selected member.  Since \(\mathcal C\) is a chain,
\(C\subsetneq C_\beta\) for every \(\beta<\alpha\); choose this \(C\) as
\(C_\alpha\).  At stage \(0\), the selected family is empty and hence is not
coinitial because \(\mathcal C\neq\varnothing\).

The recursion must stop before stage \(\kappa^+\).  Otherwise, for every
\(\alpha<\kappa^+\), choose
\[
        x_\alpha\in C_\alpha\setminus C_{\alpha+1}.
\]
These points are pairwise distinct: if \(\alpha<\beta\), then
\(x_\beta\in C_\beta\subseteq C_{\alpha+1}\), whereas
\(x_\alpha\notin C_{\alpha+1}\).  This would inject \(\kappa^+\) into
\(S\), a contradiction.  Thus the recursion stops at some
\(\theta<\kappa^+\), and the selected reverse-well-ordered subchain is
coinitial and has cardinality at most \(\kappa\).

The final assertion follows from the same choice of one point from every
successive difference.
\end{proof}

\begin{proof}[Proof of Theorem~B]
Apply Lemma~\ref{lem:chain-refinement} to the given chain
\(\mathcal B\subseteq\mathcal P(X)\).  Its coinitial reverse-well-ordered
subchain \(\mathcal D\) is still a local \(\pi\)-base: if \(U\) is a
neighbourhood of \(p\), choose \(B\in\mathcal B\) with \(B\subseteq U\),
and then \(D\in\mathcal D\) with \(D\subseteq B\).  The cardinal bounds
follow from the lemma, and \(\mathcal D\subseteq\mathcal B\) also gives
\(|\mathcal D|\leqslant|\mathcal B|\).  A strictly decreasing transfinite
sequence is, of course, both an inclusion-chain and a weakly decreasing sequence.  Conversely, the range of
a weakly decreasing transfinite sequence is an inclusion-chain and remains
a local \(\pi\)-base; applying the preceding argument to that range gives a
strictly decreasing coinitial sequence of no greater cardinality.  Thus all
three conventions yield the same invariant whenever it is defined.

Every inclusion-chain of non-empty sets is centred, which gives
\(\pcc(p,X)\leqslant\pcd(p,X)\) whenever the latter is defined.  If \(X\)
is countable, the cardinal bound gives a countable decreasing local
\(\pi\)-base and hence \(\pcd(p,X)=\omega\).  The final assertion follows.
\end{proof}

\begin{corollary}\label{cor:regular-pcd}
Whenever \(\pcd(p,X)\) is defined, it is either \(\omega\) or a regular
uncountable cardinal.  In particular, no singular uncountable cardinal,
such as \(\omega_\omega\), can occur as a decreasing local
\(\pi\)-character.
\end{corollary}

\begin{proof}
Suppose that \(\kappa=\pcd(p,X)>\omega\), and choose an inclusion-chain
local \(\pi\)-base of cardinality \(\kappa\).  By Theorem~B, it has a
coinitial strictly decreasing local \(\pi\)-base
\(\langle C_\alpha:\alpha<\theta\rangle\) with
\(|\theta|\leqslant\kappa\).  Minimality of \(\kappa\) gives
\(|\theta|=\kappa\).

The ordinal \(\theta\) cannot be a successor: in that case its last member
would by itself be a local \(\pi\)-base, forcing
\(\pcd(p,X)=\omega\).  Choose a cofinal subset of \(\theta\) of order type
\(\operatorname{cf}(\theta)\).  The corresponding subsequence is still a
local \(\pi\)-base, so minimality gives
\[
 \kappa\leqslant\operatorname{cf}(\theta)
 \leqslant|\theta|=\kappa.
\]
Consequently, \(\kappa=\operatorname{cf}(\theta)\) is regular.
\end{proof}

Let \(\mathcal F\) be a proper filter on a set \(S\).  It is \emph{free} if
\(\bigcap\mathcal F=\varnothing\).  Denote by \(X(\mathcal F)\) the space
\(S\cup\{p\}\) in which every point of \(S\) is isolated and
\[
        \bigl\{\{p\}\cup A:A\in\mathcal F\bigr\}
\]
is a neighbourhood base at \(p\).

\begin{lemma}\label{lem:filter-basics}
Let \(\mathcal F\) be a free filter on \(S\).
\begin{enumerate}
\item[(0)] The set \(S\) is infinite, the singleton \(\{p\}\) is not open
in \(X(\mathcal F)\), and every non-empty open subset of
\(X(\mathcal F)\) meets \(S\).
\item[(1)] For every finite \(a\subseteq S\), one has
\(S\setminus a\in\mathcal F\).  Consequently, every member of
\(\mathcal F\) is infinite.
\item[(2)] The space \(X(\mathcal F)\) is zero-dimensional and Hausdorff,
hence Tychonoff.
\item[(3)] Let \(\mathcal B\) be a family of non-empty open subsets of
\(X(\mathcal F)\).  Every trace \(B\cap S\), \(B\in\mathcal B\), is
non-empty, and \(\mathcal B\) is a local \(\pi\)-base at \(p\) if and only
if its trace family
\[
        \mathcal B\mathbin{|}S=\{B\cap S:B\in\mathcal B\}
\]
is coinitial in \(\mathcal F\).
\end{enumerate}
\end{lemma}

\begin{proof}
For every \(s\in S\), freeness gives \(A_s\in\mathcal F\) with
\(s\notin A_s\).  If \(a\subseteq S\) is finite, then
\(\bigcap_{s\in a}A_s\subseteq S\setminus a\), so upward closure gives
\(S\setminus a\in\mathcal F\).  Taking \(a=S\) shows that \(S\) cannot be
finite, for otherwise \(\varnothing\in\mathcal F\).  If a finite set
\(A\) belonged to \(\mathcal F\), then \(S\setminus A\in\mathcal F\) as
well, and hence
\[
        \varnothing=A\cap(S\setminus A)\in\mathcal F,
\]
contradicting properness.  This proves (1) and the first assertion of (0).

The singleton \(\{p\}\) is not open: an open set containing \(p\) contains
\(\{p\}\cup A\) for some \(A\in\mathcal F\), and such an \(A\) is
non-empty by (1).  A non-empty open set which omits \(p\) is a non-empty
subset of \(S\), while one which contains \(p\) contains a non-empty filter
member.  Thus every non-empty open set meets \(S\), completing (0).

Each basic neighbourhood \(\{p\}\cup A\) is clopen because its complement
is a subset of the discrete part \(S\).  Moreover, for every \(s\in S\),
the singleton \(\{s\}\) is clopen because
\(S\setminus\{s\}\in\mathcal F\).  These sets form a clopen base and also
separate \(p\) from each \(s\), proving (2).

The first assertion of (3) follows from (0).  If \(\mathcal B\) is local
and \(A\in\mathcal F\), choose \(B\in\mathcal B\) with
\(B\subseteq\{p\}\cup A\); then \(B\cap S\subseteq A\).  Conversely,
suppose the trace family is coinitial.  Every neighbourhood \(U\) of \(p\)
contains \(\{p\}\cup A\) for some \(A\in\mathcal F\).  Choose
\(B\in\mathcal B\) with \(B\cap S\subseteq A\).  Since
\(B\subseteq\{p\}\cup(B\cap S)\), one has \(B\subseteq U\).
\end{proof}

A family \(\mathcal A\) of infinite subsets of a set has the \emph{strong
finite-intersection property} if every finite intersection of its members is
infinite.  An infinite set \(P\) is a \emph{pseudointersection} of
\(\mathcal A\) if \(P\modfin A\) for every \(A\in\mathcal A\).  The
pseudointersection number \(\mathfrak p\) is the least cardinality of a
family of infinite subsets of a countably infinite set which has the strong
finite-intersection property but has no infinite pseudointersection; see
\cite{Blass}.

\begin{lemma}\label{lem:traces-sfip}
Let \(\mathcal F\) be a free filter on \(S\), and let \(\mathcal B\) be a
centred local \(\pi\)-base at \(p\) in \(X(\mathcal F)\).  Then
\(\mathcal B\mathbin{|}S\) is coinitial in \(\mathcal F\) and has the
strong finite-intersection property.
\end{lemma}

\begin{proof}
Coinitiality follows from Lemma~\ref{lem:filter-basics}.  The
intersection indexed by the empty subfamily is \(S\), which is infinite by
Lemma~\ref{lem:filter-basics}(0).  Let \(\mathcal B_0\) now be a non-empty
finite subfamily of \(\mathcal B\), and put
\[
        T=\bigcap_{B\in\mathcal B_0}(B\cap S).
\]
If every member of \(\mathcal B_0\) contains \(p\), then each trace belongs
to \(\mathcal F\): an open set containing \(p\) contains a basic
neighbourhood, and its trace is therefore a superset of a filter member.
Thus \(T\in\mathcal F\), and Lemma~\ref{lem:filter-basics} shows that \(T\)
is infinite.

Suppose that some member of \(\mathcal B_0\) omits \(p\).  If \(T\) were
finite, including the case \(T=\varnothing\), then
\(S\setminus T\in\mathcal F\) by Lemma~\ref{lem:filter-basics}.  Locality
would give \(B_1\in\mathcal B\) with
\[
        B_1\subseteq\{p\}\cup(S\setminus T).
\]
The intersection of \(\mathcal B_0\cup\{B_1\}\) would be empty: it cannot
contain \(p\), because a member of \(\mathcal B_0\) omits \(p\), and its
trace on \(S\) is contained in \(T\cap(S\setminus T)=\varnothing\).  This
contradicts centredness.  Hence \(T\) is infinite.
\end{proof}

\begin{lemma}\label{lem:countable-pseudo}
Every non-empty countable family of subsets of a countable set with the
strong finite-intersection property has an infinite pseudointersection.
Consequently, \(\mathfrak p\geqslant\omega_1\).
\end{lemma}

\begin{proof}
Enumerate the family as \(\{A_n:n<\omega\}\), repeating members when the
family is finite.  Recursively choose distinct points
\[
        x_n\in\bigcap_{k\leqslant n}A_k.
\]
This is possible because every displayed intersection is infinite.  Then
\(P=\{x_n:n<\omega\}\) is infinite and
\(P\setminus A_k\subseteq\{x_0,\ldots,x_{k-1}\}\) for every \(k<\omega\).
\end{proof}

Let \(\mathcal F_0\) and \(\mathcal F_1\) be free filters on disjoint sets
\(S_0\) and \(S_1\).  Their sum is
\[
 \mathcal F_0\oplus\mathcal F_1
 =\bigl\{A\subseteq S_0\dcup S_1:
 A\cap S_i\in\mathcal F_i\text{ for }i<2\bigr\}.
\]

\begin{lemma}\label{lem:sum-free}
The family \(\mathcal F_0\oplus\mathcal F_1\) is a free filter on
\(S_0\dcup S_1\).
\end{lemma}

\begin{proof}
The filter axioms follow coordinatewise, and properness follows from the
properness of the two coordinate filters.  If \(s\in S_i\), choose
\(A_i\in\mathcal F_i\) with \(s\notin A_i\), and choose any
\(A_{1-i}\in\mathcal F_{1-i}\).  Then
\(A_0\cup A_1\in\mathcal F_0\oplus\mathcal F_1\) and omits \(s\).  Hence
the intersection of the sum filter is empty.
\end{proof}

\begin{lemma}\label{lem:sum-projection}
Let \(\mathcal B\) be a centred local \(\pi\)-base at the non-isolated
point of \(X(\mathcal F_0\oplus\mathcal F_1)\).  For some \(i<2\),
\[
        \mathcal B_i=\{B\cap S_i:B\in\mathcal B\}
\]
is a centred local \(\pi\)-base at the non-isolated point of
\(X(\mathcal F_i)\).  If \(\mathcal B\) is linearly ordered by inclusion,
then \(\mathcal B_i\) is linearly ordered by inclusion as well.
\end{lemma}

\begin{proof}
By Lemmas~\ref{lem:sum-free} and~\ref{lem:traces-sfip}, the traces of
\(\mathcal B\) on \(S_0\dcup S_1\) have the strong finite-intersection
property.  At least one coordinate family has the finite-intersection
property.  Indeed, otherwise there would be finite subfamilies
\(\mathcal B'_0,\mathcal B'_1\subseteq\mathcal B\) such that
\[
 \bigcap_{B\in\mathcal B'_i}(B\cap S_i)=\varnothing
 \qquad(i<2).
\]
Both \(\mathcal B'_0\) and \(\mathcal B'_1\) are necessarily non-empty,
because the intersection indexed by the empty family is
\(S_i\neq\varnothing\).  Thus
\(\mathcal B'=\mathcal B'_0\cup\mathcal B'_1\) is a legitimate non-empty
finite subfamily.  The intersection of its traces is empty in each
coordinate, and hence
\[
 \bigcap_{B\in\mathcal B'}\bigl(B\cap(S_0\dcup S_1)\bigr)=\varnothing,
\]
contrary to Lemma~\ref{lem:traces-sfip}.

Fix a coordinate \(i\) whose projection has the finite-intersection
property.  In particular, every member of \(\mathcal B_i\) is non-empty,
and \(\mathcal B_i\) is centred.  Given \(A_i\in\mathcal F_i\), choose any
\(A_{1-i}\in\mathcal F_{1-i}\).  Locality in the sum yields
\(B\in\mathcal B\) with
\[
        B\subseteq\{p\}\cup A_0\cup A_1.
\]
Hence \(\varnothing\neq B\cap S_i\subseteq A_i\), proving locality in the
\(i\)-th component.  Finally, coordinate projection preserves inclusions.
\end{proof}

\section{The countable star filter}\label{sec:star}

Let \(E=[\omega]^2\), viewed as the edge set of the complete graph on
\(\omega\).  For \(n<\omega\), let
\[
        L_n=\bigl\{\{n,m\}:m<\omega,\ m\neq n\bigr\}
\]
be the star at \(n\).  Fix a non-principal ultrafilter \(U\) on \(\omega\).
For \(A\subseteq E\), put
\[
        \sigma(A)=\{n<\omega:L_n\subseteq A\},
        \qquad
        \FU=\{A\subseteq E:\sigma(A)\in U\}.
\]

\begin{lemma}\label{lem:star-filter}
The family \(\FU\) is a free filter on \(E\), and
\(\{L_n:n<\omega\}\) is a linked family coinitial in \(\FU\).
\end{lemma}

\begin{proof}
The identities
\[
        \sigma(A\cap B)=\sigma(A)\cap\sigma(B)
        \quad\text{and}\quad
        A\subseteq B\Longrightarrow\sigma(A)\subseteq\sigma(B)
\]
show that \(\FU\) is a filter; it is proper because
\(\sigma(\varnothing)=\varnothing\notin U\).  If \(a\subseteq E\) is
finite, then every star except those centred at one of the finitely many
endpoints of edges in \(a\) is contained in \(E\setminus a\).  A
non-principal ultrafilter on \(\omega\) contains every cofinite subset of
\(\omega\), so \(\sigma(E\setminus a)\in U\), and hence
\(E\setminus a\in\FU\).  In particular,
\[
 \bigcap\FU\subseteq\bigcap_{e\in E}(E\setminus\{e\})=\varnothing,
\]
which proves freeness.

For distinct \(n,m<\omega\),
\(L_n\cap L_m=\{\{n,m\}\}\), and hence the stars are linked.  If
\(A\in\FU\), then \(\sigma(A)\in U\) is non-empty; choosing
\(n\in\sigma(A)\) gives \(L_n\subseteq A\).  Thus the stars are coinitial.
\end{proof}

\begin{lemma}\label{lem:graph-ultrafilter}
For every infinite \(P\subseteq E\), there is a set \(B\subseteq\omega\)
with \(B\notin U\) such that \(P\cap[B]^2\) is infinite.
\end{lemma}

\begin{proof}
Regard \(P\) as a graph on \(\omega\).  Suppose first that some vertex
\(n\) has infinitely many neighbours.  Partition an infinite set of these
neighbours into two disjoint infinite sets.  Choose one part \(M\) which
does not belong to \(U\).  (This is possible because the two parts are
disjoint.)  Since members
of \(E=[\omega]^2\) have two distinct endpoints, \(n\) is not one of its
own neighbours, and hence \(n\notin M\).  Put \(B=M\cup\{n\}\).  If
\(B\in U\), then
\(\omega\setminus M\in U\) as \(M\notin U\), and hence
\[
        \{n\}=B\cap(\omega\setminus M)\in U,
\]
contrary to non-principality.  Thus \(B\notin U\), while
\(P\cap[B]^2\) contains infinitely many edges incident with \(n\).

Suppose now that every vertex has finite degree in \(P\).  A greedy
construction produces an infinite matching: after choosing finitely many
matching edges, only finitely many edges are incident with their endpoints,
so infinitely many edges remain.  Split the matching into two infinite
submatchings.  Their vertex sets are disjoint and therefore cannot both
belong to \(U\).  The vertex set \(B\) of one which does not belong to
\(U\) satisfies the conclusion.
\end{proof}

\begin{lemma}\label{lem:no-pseudo-star}
The filter \(\FU\) has no infinite pseudointersection.
\end{lemma}

\begin{proof}
Let \(P\subseteq E\) be infinite.  Choose \(B\notin U\) as in
Lemma~\ref{lem:graph-ultrafilter}, and put \(A=\omega\setminus B\in U\).
Set
\[
        H_A=\bigcup_{n\in A}L_n.
\]
Since \(A\subseteq\sigma(H_A)\), we have \(H_A\in\FU\).  On the other
hand,
\[
        P\setminus H_A=P\cap[B]^2
\]
is infinite.  Hence \(P\nmodfin H_A\).  As \(P\) was arbitrary, no
infinite pseudointersection exists.
\end{proof}

\begin{proposition}\label{prop:star-invariants}
Let \(Y=X(\FU)\), and let \(p\) be its non-isolated point.  Then
\[
        \pcl(p,Y)=\omega,
        \qquad
        \pcc(p,Y)\geqslant\mathfrak p,
\]
and \(p\) admits no decreasing local \(\pi\)-base, under any of the three conventions
from Theorem~B.
\end{proposition}

\begin{proof}
Lemma~\ref{lem:star-filter} gives a countable linked local \(\pi\)-base, so
\(\pcl(p,Y)=\omega\).

Suppose that \(\mathcal C\) is a centred local \(\pi\)-base at \(p\) with
\(|\mathcal C|<\mathfrak p\).  By Lemma~\ref{lem:traces-sfip}, its trace
family on \(E\) is coinitial in \(\FU\), has the strong
finite-intersection property, and has cardinality below \(\mathfrak p\).
It therefore has an infinite pseudointersection \(P\).  If \(A\in\FU\),
coinitiality gives a trace \(T\subseteq A\), and
\(P\modfin T\subseteq A\).  Thus \(P\) is a pseudointersection of
\(\FU\), contradicting Lemma~\ref{lem:no-pseudo-star}.  Hence
\(\pcc(p,Y)\geqslant\mathfrak p\).

By Lemma~\ref{lem:countable-pseudo}, \(\mathfrak p\geqslant\omega_1\), so
\(\pcc(p,Y)>\omega\).  Since \(Y\) is countable, Theorem~B shows that no
decreasing local \(\pi\)-base can exist at \(p\).
\end{proof}

\begin{remark}\label{rem:why-matrix}
The estimate by \(\mathfrak p\) explains the second component of the final
construction.  The value of \(\mathfrak p\) is not fixed by ZFC and may
exceed \(\omega_1\); see \cite{Blass}.  Hence the star space cannot by
itself have centred local \(\pi\)-character exactly \(\omega_1\) in every
model of ZFC.  It is used only to supply the linked base and the obstruction
to decreasing bases, while the matrix component fixes the exact centred
value.
\end{remark}

\section{The matrix filter}\label{sec:matrix}

Each member of the set \(D\) defined below has a dual role: it is both a
point of the discrete part of a one-point filter space and a subset of the
auxiliary set \(S\).  Intersections such as \(K_P=\bigcap_{d\in T_P}d\)
use the latter interpretation.

Let
\[
        S=\omega_1\times\omega_2.
\]
For \(\alpha<\omega_2\) and \(F\in[\omega_1]^{<\omega}\), define
\[
 d_{\alpha,F}
   =(\omega_1\times\alpha)\cup(F\times\omega_2),
 \qquad
 D=\{d_{\alpha,F}:\alpha<\omega_2,\ F\in[\omega_1]^{<\omega}\}.
\]

\begin{remark}\label{rem:unique-matrix}
The representation of a member of \(D\) is unique.  Indeed,
\[
 F=\bigl\{\xi<\omega_1:
       \{\xi\}\times\omega_2\subseteq d_{\alpha,F}\bigr\}.
\]
Once \(F\) is known, choose \(\xi\in\omega_1\setminus F\), which is
possible because \(F\) is finite and \(\omega_1\) is uncountable; the
\(\xi\)-section of \(d_{\alpha,F}\) is exactly the ordinal \(\alpha\).
Thus \(d_{\alpha,F}=d_{\beta,H}\) implies \(F=H\) and \(\alpha=\beta\).
In particular,
\[
        |D|=\omega_2.
\]
\end{remark}

For finite \(e\subseteq S\), put
\[
        G_e=\{d\in D:e\subseteq d\},
\]
and let
\[
        \G=\{A\subseteq D:G_e\subseteq A
              \text{ for some }e\in[S]^{<\omega}\}.
\]

\begin{lemma}\label{lem:matrix-filter}
The family \(\G\) is a free filter on \(D\).
\end{lemma}

\begin{proof}
For every finite \(e\subseteq S\), the set \(G_e\) is non-empty.  Choose
\(\alpha<\omega_2\) such that every second coordinate occurring in \(e\)
is strictly below \(\alpha\) (and take \(\alpha=0\) when \(e=\varnothing\));
then \(d_{\alpha,\varnothing}\in G_e\).  Moreover,
\[
        G_{e_0}\cap G_{e_1}=G_{e_0\cup e_1},
\]
so the sets \(G_e\) generate a proper filter.

To prove freeness, let
\[
        a=\{d_{\alpha_j,F_j}:j<k\}\subseteq D
\]
be finite; the case \(k=0\) is immediate.  The notation is unambiguous by
Remark~\ref{rem:unique-matrix}.  Choose
\(\xi\in\omega_1\setminus\bigcup_{j<k}F_j\) and
\(\beta<\omega_2\) with \(\beta\geqslant\alpha_j\) for every \(j<k\).
Then \((\xi,\beta)\notin d_{\alpha_j,F_j}\) for every \(j<k\), whence
\[
        G_{\{(\xi,\beta)\}}\subseteq D\setminus a.
\]
Thus \(D\setminus a\in\G\).  Taking \(a=\{d\}\) successively for
\(d\in D\), we obtain
\[
 \bigcap\G\subseteq\bigcap_{d\in D}(D\setminus\{d\})=\varnothing.
\]
Hence \(\G\) is free.
\end{proof}

Let \(Z=X(\G)\), and denote its non-isolated point by \(q\).

\begin{lemma}\label{lem:matrix-centred}
One has
\[
        \pc(q,Z)=\pcl(q,Z)=\pcc(q,Z)=\omega_1.
\]
In particular, the lower bound holds for arbitrary local \(\pi\)-bases, not
only for linked or centred ones.
\end{lemma}

\begin{proof}
For \(F\in[\omega_1]^{<\omega}\), put
\[
        B_F=\{d_{\alpha,H}\in D:F\subseteq H\}.
\]
This is well defined by Remark~\ref{rem:unique-matrix}.  The family
\(\mathcal B=\{B_F:F\in[\omega_1]^{<\omega}\}\) is centred, since
\[
        B_{F_0}\cap\cdots\cap B_{F_{n-1}}
        =B_{F_0\cup\cdots\cup F_{n-1}}\neq\varnothing.
\]
It is coinitial in \(\G\).  Indeed, if \(A\in\G\), choose finite
\(e\subseteq S\) with \(G_e\subseteq A\), and let \(F_e\) be the finite set
of first coordinates occurring in \(e\).  Then
\[
        B_{F_e}\subseteq G_e\subseteq A.
\]
Consequently, \(\pcc(q,Z)\leqslant\omega_1\).

It remains to rule out smaller local \(\pi\)-bases.  Let \(\mathcal P\) be
a non-empty family of non-empty open subsets of \(Z\) with
\(|\mathcal P|<\omega_1\).  Every \(P\in\mathcal P\) meets \(D\): this is
clear if \(q\notin P\), while if \(q\in P\), then \(P\) contains a basic
neighbourhood \(\{q\}\cup G_e\), and \(G_e\neq\varnothing\) by
Lemma~\ref{lem:matrix-filter}.  Choose
\(d_{\alpha_P,F_P}\in P\cap D\) for every \(P\in\mathcal P\).

The set \(\bigcup_{P\in\mathcal P}F_P\) is countable, so choose
\[
        \xi\in\omega_1\setminus\bigcup_{P\in\mathcal P}F_P.
\]
Since \(\mathcal P\) is countable and
\(\operatorname{cf}(\omega_2)>\omega\), choose \(\beta<\omega_2\) with
\(\beta\geqslant\alpha_P\) for every
\(P\in\mathcal P\).  Then
\((\xi,\beta)\notin d_{\alpha_P,F_P}\) for every \(P\), and hence no member
of \(\mathcal P\) is contained in the neighbourhood
\[
        \{q\}\cup G_{\{(\xi,\beta)\}}.
\]
Thus \(\mathcal P\) is not a local \(\pi\)-base.  Therefore
\(\pc(q,Z)\geqslant\omega_1\), and the displayed equalities follow.
\end{proof}

\begin{lemma}\label{lem:matrix-decreasing}
One has
\[
        \pcd(q,Z)=\omega_2.
\]
More precisely, every inclusion-chain which is a local \(\pi\)-base at
\(q\) has cardinality at least \(\omega_2\).
\end{lemma}

\begin{proof}
For \(\alpha<\omega_2\), let
\[
        C_\alpha=\{d_{\beta,F}\in D:\beta\geqslant\alpha\}.
\]
The family \(\langle C_\alpha:\alpha<\omega_2\rangle\) is decreasing and
coinitial in \(\G\).  Indeed, if \(A\in\G\), choose finite \(e\subseteq S\)
with \(G_e\subseteq A\), and choose \(\alpha\) strictly larger than every
second coordinate occurring in \(e\).  Then
\(C_\alpha\subseteq G_e\subseteq A\).  Hence
\(\pcd(q,Z)\leqslant\omega_2\).

Suppose that \(\mathcal P\) is an inclusion-chain which is a local
\(\pi\)-base at \(q\) and satisfies \(|\mathcal P|<\omega_2\).  For
\(P\in\mathcal P\), set
\[
        T_P=P\cap D,
        \qquad
        K_P=\bigcap_{d\in T_P}d\subseteq S.
\]
Every \(T_P\) is non-empty: if \(q\notin P\), then
\(P=T_P\neq\varnothing\); if \(q\in P\), openness gives a basic
neighbourhood \(\{q\}\cup G_e\subseteq P\), and
\(G_e\neq\varnothing\).  Locality gives
\begin{equation}\label{eq:matrix-cover}
        S=\bigcup_{P\in\mathcal P}K_P.
\end{equation}
Indeed, for \(s\in S\), choose \(P\in\mathcal P\) with
\(P\subseteq\{q\}\cup G_{\{s\}}\).  Then every member of \(T_P\) contains
\(s\), so \(s\in K_P\).

For \(P\in\mathcal P\), define the set of full columns of \(K_P\) by
\[
        F_P^*=\{\xi<\omega_1:\{\xi\}\times\omega_2\subseteq K_P\}.
\]
Each \(F_P^*\) is finite: choosing any
\(d_{\alpha,F}\in T_P\), one has \(K_P\subseteq d_{\alpha,F}\), and only
columns indexed by \(F\) are full in \(d_{\alpha,F}\).  The map
\(P\mapsto F_P^*\) reverses inclusion.  Indeed, if \(P\subseteq P'\), then
\(T_P\subseteq T_{P'}\), whence \(K_P\supseteq K_{P'}\) and
\(F_P^*\supseteq F_{P'}^*\).  Thus
\(\{F_P^*:P\in\mathcal P\}\) is again a chain.  A chain of finite sets has
at most one distinct member of each finite cardinality, and therefore only
countably many distinct members.  Hence
\[
        F^*=\bigcup_{P\in\mathcal P}F_P^*
\]
is countable.

Choose \(\xi\in\omega_1\setminus F^*\).  For every
\(P\in\mathcal P\), there is some \(d_{\alpha_P,F_P}\in T_P\) with
\(\xi\notin F_P\); otherwise every member of \(T_P\) would contain the
whole \(\xi\)-th column, contrary to \(\xi\notin F_P^*\).  Therefore
\[
        K_P\cap(\{\xi\}\times\omega_2)
        \subseteq\{\xi\}\times\alpha_P.
\]
Since \(|\mathcal P|<\omega_2\) and \(\omega_2\) is regular,
\[
        \beta=\sup_{P\in\mathcal P}\alpha_P<\omega_2.
\]
Then \((\xi,\beta)\notin K_P\) for every \(P\in\mathcal P\), contradicting
\eqref{eq:matrix-cover}.  Thus no inclusion-chain local \(\pi\)-base of
cardinality below \(\omega_2\) exists.
\end{proof}

\section{The ZFC construction}

Fix a non-principal ultrafilter \(U\) on \(\omega\), and take disjoint
copies of the sets \(E\) and \(D\) from Sections~\ref{sec:star}
and~\ref{sec:matrix}.  Form the corresponding star filter \(\FU\) and
matrix filter \(\G\), and put
\[
        \Fsum=\FU\oplus\G,
        \qquad
        X=X(\Fsum)=(E\dcup D)\cup\{p\}.
\]
By Lemma~\ref{lem:sum-free}, \(\Fsum\) is a free filter.  Hence
Lemma~\ref{lem:filter-basics} shows that \(X\) is zero-dimensional and
Tychonoff.  Moreover, \(|X|=\omega_2\) by
Remark~\ref{rem:unique-matrix}, and \(p\) is its only non-isolated point.

\begin{proof}[Proof of Theorem~A]
The stars \(L_n\subseteq E\) form a countable linked local \(\pi\)-base at
\(p\) in \(X\).  Indeed, the \(E\)-trace of every neighbourhood of \(p\)
belongs to \(\FU\), and hence contains a star.  Thus
\[
        \pc(p,X)=\pcl(p,X)=\omega.
\]

The family \(\{B_F:F\in[\omega_1]^{<\omega}\}\) from
Lemma~\ref{lem:matrix-centred}, regarded as a family of subsets of \(D\), is
a centred local \(\pi\)-base at \(p\) in \(X\), because the
\(D\)-trace of every neighbourhood of \(p\) belongs to \(\G\).  Hence
\(\pcc(p,X)\leqslant\omega_1\).  If a countable centred local
\(\pi\)-base existed at \(p\), Lemma~\ref{lem:sum-projection} would project
it to a countable centred local \(\pi\)-base in either \(X(\FU)\) or
\(X(\G)\).  The first alternative contradicts
Proposition~\ref{prop:star-invariants}, because
\(\mathfrak p\geqslant\omega_1\), and the second contradicts
Lemma~\ref{lem:matrix-centred}.  Therefore
\[
        \pcc(p,X)=\omega_1.
\]

The family \(\langle C_\alpha:\alpha<\omega_2\rangle\) from
Lemma~\ref{lem:matrix-decreasing}, again regarded as a family of subsets of
\(D\), is a decreasing local \(\pi\)-base at \(p\), for the same
reason.  Thus
\(\pcd(p,X)\leqslant\omega_2\).  Suppose that an inclusion-chain local
\(\pi\)-base of cardinality below \(\omega_2\) existed at \(p\).  Such a
chain is centred, so Lemma~\ref{lem:sum-projection} would project it to a
decreasing local \(\pi\)-base in \(X(\FU)\) or \(X(\G)\).  The star
component admits no decreasing local \(\pi\)-base at all, by
Proposition~\ref{prop:star-invariants}; the matrix component admits none of
cardinality below \(\omega_2\), by Lemma~\ref{lem:matrix-decreasing}.  This
is impossible, and consequently
\[
        \pcd(p,X)=\omega_2.
\]

Every point \(x\neq p\) is isolated, and \(\{\{x\}\}\) is a linked,
centred, and decreasing local \(\pi\)-base at \(x\).  Thus all four
pointwise invariants at \(x\) are \(\omega\), and in particular
\(\pcd(x,X)\) is defined at every point.  The displayed equalities at \(p\)
therefore give
\[
 \pc(X)=\pcl(X)=\omega<\pcc(X)=\omega_1<\pcd(X)=\omega_2.\qedhere
\]
\end{proof}

\begin{remark}\label{rem:optimal-cardinality}
The cardinality in Theorem~A is optimal.  Indeed, Theorem~B gives
\[
        \pcd(x,W)\leqslant\max\{\omega,|W|\}
\]
for every \(x\in W\) at which \(\pcd(x,W)\) is defined.  Consequently, a
space satisfying \(\pcd(W)=\omega_2\) must have cardinality at least
\(\omega_2\), while the space constructed above has exactly that
cardinality.
\end{remark}

\begin{remark}\label{rem:higher-cardinals}
The matrix construction generalises verbatim.  Let \(\kappa<\lambda\) be
regular uncountable cardinals, replace \(S\) by \(\kappa\times\lambda\),
and put
\[
 d_{\alpha,F}=(\kappa\times\alpha)\cup(F\times\lambda)
 \qquad
 (\alpha<\lambda,\ F\in[\kappa]^{<\omega}).
\]
The resulting matrix component has ordinary, linked, and centred local
\(\pi\)-character \(\kappa\), and decreasing local \(\pi\)-character
\(\lambda\).  Its sum with the star component therefore gives
\[
 \pc(X)=\pcl(X)=\omega,
 \qquad \pcc(X)\leqslant\kappa,
 \qquad \pcd(X)=\lambda.
\]
The projection argument yields the reverse inequality
\(\pcc(X)\geqslant\kappa\) whenever \(\kappa\leqslant\mathfrak p\).
Without this hypothesis, the star projection is not excluded for families
of cardinality in \([\mathfrak p,\kappa)\).  Since ZFC proves
\(\omega_1\leqslant\mathfrak p\), but does not prove
\(\kappa\leqslant\mathfrak p\) for any fixed regular
\(\kappa>\omega_1\), the present method gives an unconditional exact
centred value only when \(\kappa=\omega_1\); see \cite{Blass}.
\end{remark}

\begin{question}\label{que:higher-cardinals}
Let \(\omega<\kappa<\lambda\) be regular cardinals.  Is there in ZFC a
Hausdorff space \(X\) such that
\[
 \pc(X)=\pcl(X)=\omega<\pcc(X)=\kappa<\pcd(X)=\lambda?
\]
\end{question}

\end{document}